\documentclass[pdflatex,sn-mathphys-num]{sn-jnl}

\usepackage{graphicx}%
\usepackage{multirow}%
\usepackage{amsmath,amssymb,amsfonts}%
\usepackage{amsthm}%
\usepackage{mathrsfs}%
\usepackage[title]{appendix}%
\usepackage{xcolor}%
\usepackage{textcomp}%
\usepackage{manyfoot}%
\usepackage{booktabs}%
\usepackage{algorithm}%
\usepackage{algorithmicx}%
\usepackage{algpseudocode}%
\usepackage{listings}%
\usepackage{tikz}
\usetikzlibrary{cd}
\usetikzlibrary{decorations.pathmorphing, positioning}
\usepackage{subcaption}
\usepackage{arydshln}
\usepackage{adjustbox}

\theoremstyle{thmstyleone}%
\newtheorem{theorem}{Theorem}%
\newtheorem{proposition}[theorem]{Proposition}%
\newtheorem{corollary}[theorem]{Corollary}

\theoremstyle{thmstyletwo}%
\newtheorem{example}{Example}%
\newtheorem{remark}{Remark}%

\theoremstyle{thmstylethree}%
\newtheorem{definition}{Definition}%
\newtheorem{assumption}{Assumption}
\newtheorem{problem}{Problem}

\begin{document}
	
	\title[Feedback Synthesis]{A Dynamical Systems view of Feedback Synthesis}
	
	\author*[1]{\fnm{William} \sur{Clark}}\email{clarkw3@ohio.edu}
	
	\affil*[1]{\orgdiv{Department of Mathematics}, \orgname{Ohio University}, \orgaddress{\street{24 Race Street}, \city{Athens}, \postcode{45701}, \state{Ohio}, \country{USA}}}
	
	
	\abstract{Control theory and dynamical systems are closely intertwined fields. Pontryagin's maximum principal offers a strong connection by providing a constructive way to synthesize feedback control laws via lifting the control system to a (Hamiltonian) dynamical system. Feedback laws can then be encoded as invariant manifolds of this induced dynamical system - under the condition that these manifolds project diffeomorphically back to the base control system.
		
		While feedback stabilization is impossible for many control systems, the above procedure can still be carried out. In this setting, the invariant manifold no longer projects diffeomorphically which results in the emergence of caustics.
		
		This paper offers an overview of the above connection by translating target sets from controls to isotropic submanifolds in symplectic geometry and associates feedback controllability to singularities of the induced invariant manifolds. Multiple low-dimensional examples are included to elucidate the theory.}
	
	\keywords{Feedback control, Lagrangian submanifolds, Invariant manifold theory for dynamical systems}
	
	\pacs[MSC Classification]{93B52, 53D12, 34C45}
	
	\maketitle
	\section{Introduction}
	The field of dynamical systems can be roughly described as ``things changing in time'' where its fundamental goal is to understand the long-term behavior of trajectories, e.g., \cite[Chapter~0]{katok_hasselblatt}. On the other hand, control theory begins with a desired long-term behavior and attempts to synthesize a dynamical system that realizes it, e.g., \cite{sontag}. In this sense, control theory can be thought of as the inverse problem of dynamical systems.
	\begin{center}
		\begin{tikzcd}
			\text{Control Theory}\arrow[bend left=25]{rr} & & \text{Dynamical Systems}\arrow[bend left=25]{ll}
		\end{tikzcd}
	\end{center}
	This association allows for results/insights from one discipline to have applications in the other. There is a long history of dynamical results being used in control. A (very non-exhaustive) list includes:
	\begin{enumerate}
		\item Lyapunov's (direct) method and LaSalle's invariant principle \cite{lasalle}.
		\item Hartman-Grobman and the stable/unstable manifold theorems \cite[Chapter~6]{katok_hasselblatt}.
		\item Structural stability \cite[Chatper~3]{geometric_arnold}.
		\item Koopman theory \cite{koopman}.
	\end{enumerate}
	Dynamical systems, on the other hand, has relied substantially less on control theory for its development. The purpose of this work is to provide a case where dynamical systems theory can be furthered by control theory. This will be accomplished by translating the problem of finding a stabilizing feedback control into a purely dynamical systems problem. This approach offers three insights:
	\begin{enumerate}
		\item Results in control theory can be translated to dynamical systems, e.g., Brockett's condition in control theory gets translated to catastrophies in dynamical systems.
		\item Target sets are associated to isotropic submanifolds of the cotangent bundle over the control system. 
		\item The induced dynamical systems provide interesting and non-trivial examples, e.g., the nonholonomic integrator generates a system with an infinite cascade of cusp catastrophies.
	\end{enumerate}
	The bridge between controls and dynamics studied here will be Pontryagin's maximum principle (PMP) \cite{Pontryagin1962}. This lifts an optimal control problem into a Hamiltonian system on the cotangent bundle of the state-space which allows for a symplectic and dynamical interpretation of feedback laws. These correspond to:
	\begin{description}
		\item[Symplectic geometry:] Hamiltonian-invariant Lagrangian submanifolds, and
		\item[Dynamical systems:] stable manifolds.
	\end{description}
	See Fig. \ref{fig:algorithm_overview} for an overview the procedure in the special case of stabilizing a fixed point $\tilde{x}$. Synthesizing a feedback via this procedure will be referred to as \textit{characteristic feedback synthesis}.
	
	\begin{figure}
		\centering
		\begin{adjustbox}{frame}
			\begin{tikzcd}
				\text{Feedback Control} \arrow[r, leftrightsquigarrow] & \text{Symplectic Geometry} \arrow[r, leftrightsquigarrow] & \text{Dynamical Systems} \\[2ex]
				\text{Target set} \arrow[u, phantom, "\subset"{rotate=90}]\arrow[r]\arrow[d, dashed] & \text{Isotropic Submanifold} \arrow[u, phantom, "\subset"{rotate=90}]
				\arrow[r] 
				& \text{Stable Manifold} \arrow[u, phantom, "\subset"{rotate=90}]\arrow[dll, bend left=10]\\
				\text{Feedback law}
			\end{tikzcd}
		\end{adjustbox}
		\medskip
		
		\fbox{
			\begin{tikzpicture}[
				node distance=1.5cm,
				box/.style={
					draw,
					rectangle,
					minimum width=3cm,
					minimum height=1cm,
					align=center
				}
				]
				\node[box] (a) {$\dot{x} = f(x, u), \quad J = \displaystyle\int_0^\infty \, \ell(x, u) \, dt$};
				\node[box, below=of a] (b) {$\begin{array}{rl}
						\dot{x} &= f(x, \mu(x,p)) \\[1ex]
						\dot{p} &= \dfrac{\partial \ell}{\partial x}(x, \mu(x,p)) - p\dfrac{\partial f}{\partial x}(x, \mu(x,p)) \\[2ex]
						\mu(x,p) &= \displaystyle\arg\max_{\nu \in \mathcal{U}} \ \left[ \langle p, f(x,u)\rangle - \ell(x,u) \right]
					\end{array}$};
				\node[box, below=of b] (c) {$\Lambda = \left\{ (x,p) : \displaystyle\lim_{t\to\infty} \, \varphi_t(x,p) \to (\tilde{x}, \star) \right\}$};
				\draw[->] (a) -- (b);
				\draw[->] (b) -- (c);
				\draw[->, dashed] (c.east) to [out=0, in=0, looseness=1] (a.east);
				\node[left=4.35cm of a.center] {\shortstack{Cost \\ function}};
				\node[left=4cm of b.center] {\shortstack{Hamiltonian \\ system}};
				\node[left=4.3cm of c.center] {\shortstack{Stable \\ manifold}};
		\end{tikzpicture}}
		\caption{Overview of the characteristic feedback synthesis.}
		\label{fig:algorithm_overview}
	\end{figure}
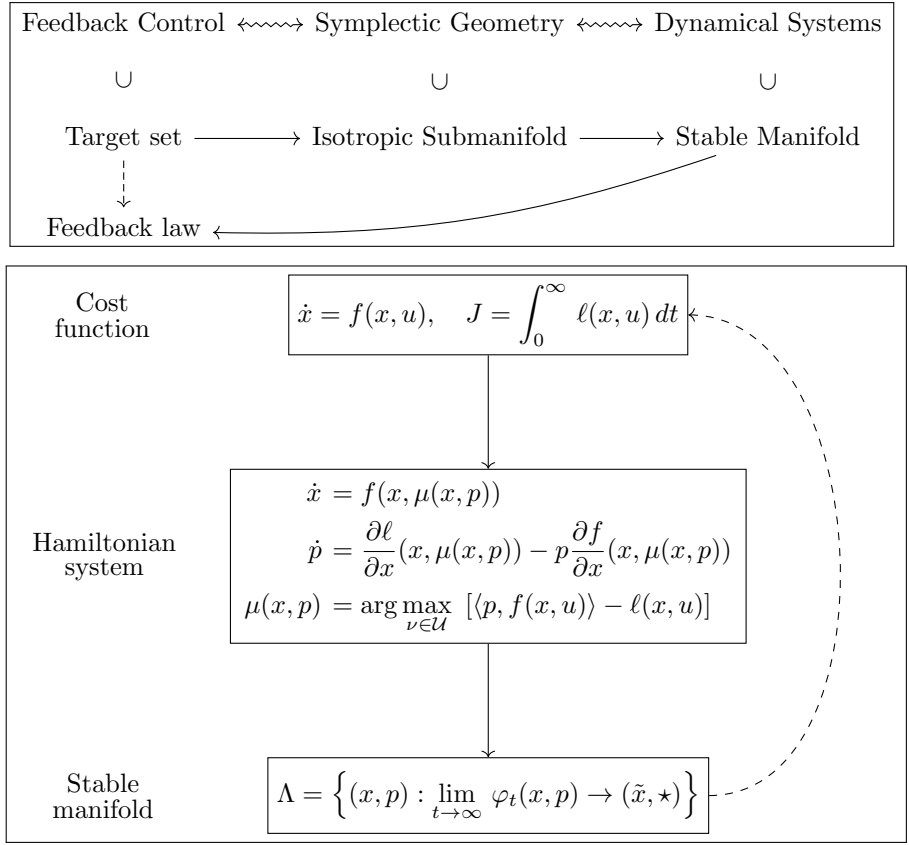
	Suppose that the resulting stable manifold $\Lambda$ is the graph of a function, i.e. there exists a (smooth) function $\rho$ such that
	\begin{equation*}
		\Lambda = \left\{ (x, \rho(x)) : x\in Q \right\}.
	\end{equation*}
	A stable dynamical system can be created by the feedback $u = \mu(x, \rho(x))$:
	\begin{equation*}
		\dot{x} = f\left(x, \mu(x, \rho(x))\right).
	\end{equation*}
	This results in the following insight (see Fig. \ref{fig:example_manifolds}):
	\begin{equation}\label{eq:insight}
		\begin{tikzcd}
			\text{stabilizing feedback} \arrow[r, leftrightarrow] & \text{stable manifold is a graph}.
		\end{tikzcd}
	\end{equation}
	\begin{figure}
		\centering
		\begin{subfigure}{0.475\textwidth}
			\centering
			\begin{tikzpicture}
				\draw[thick, ->] (-3,0) -- (3,0);
				\draw[thick, ->] (0,-3) -- (0,3);
				\draw[very thick, blue, ->] (3,2) to [out=200, in=20] (0.05,0.025);
				\draw[very thick, blue, ->] (-3,-2) to [out=10, in=200] (-0.05,-0.025);
				\draw[very thick, red, ->, domain=0:3, smooth] plot ({-0.75*sin(deg(1.5*\x))}, \x); 
				\draw[very thick, red, ->, domain=0:-3, smooth] plot ({-0.75*sin(deg(1.5*\x))}, \x); 
				\node[blue] at (2.5,1) {$\Lambda_1 = W^s_1$};
				\node[red] at (0.5,2) {$W^u_1$};
				\node[below] at (3,-0.05) {$x$};
				\node[left] at (0,3) {$p$};
				\node[below left] at (0,-0.1) {$\tilde{x}$};
			\end{tikzpicture}
			\caption{A stabilizing feedback law exists.}
		\end{subfigure}
		\begin{subfigure}{0.475\textwidth}
			\centering
			\begin{tikzpicture}
				\draw[thick, ->] (-3,0) -- (3,0);
				\draw[thick, ->] (0,-3) -- (0,3);
				\draw[very thick, blue, ->, domain=2.55:0.05, smooth] plot ({1.5*\x*\x*(1-\x/2)}, \x);
				\draw[very thick, blue, ->, domain=-1.3:-0.05, smooth] plot ({\x*\x*(1-\x/2)}, \x);
				\draw[very thick, red, ->, domain=0:2.75, smooth] plot (\x, {\x*\x/3});
				\draw[very thick, red, ->, domain=0:-2.75, smooth] plot (\x, {-\x*\x/3});
				\node[blue] at (2,-1.6) {$\Lambda_2 = W^s_2$};
				\node[red] at (2.55,1.4) {$W^u_2$};
				\node[below] at (3,-0.05) {$x$};
				\node[left] at (0,3) {$p$};
				\node[below left] at (0,-0.1) {$\tilde{x}$};
			\end{tikzpicture}
			\caption{The induced feedback law does not exist.}
			\label{subfig:caustic_feedback}
		\end{subfigure}
		\caption{Left: A non-singular manifold as $\tilde{x}\not\in \Sigma(\Lambda_1)$. This set can be expressed as the graph of a function $p=\rho(x)$ which induces the feedback law $k(x)=\mu(x, \rho(x))$. Right: A singular manifold as $\tilde{x}\in\Sigma(\Lambda_2)$. This set cannot be expressed as a graph of a function and the induced feedback law will not be continuous.}
		\label{fig:example_manifolds}
	\end{figure}
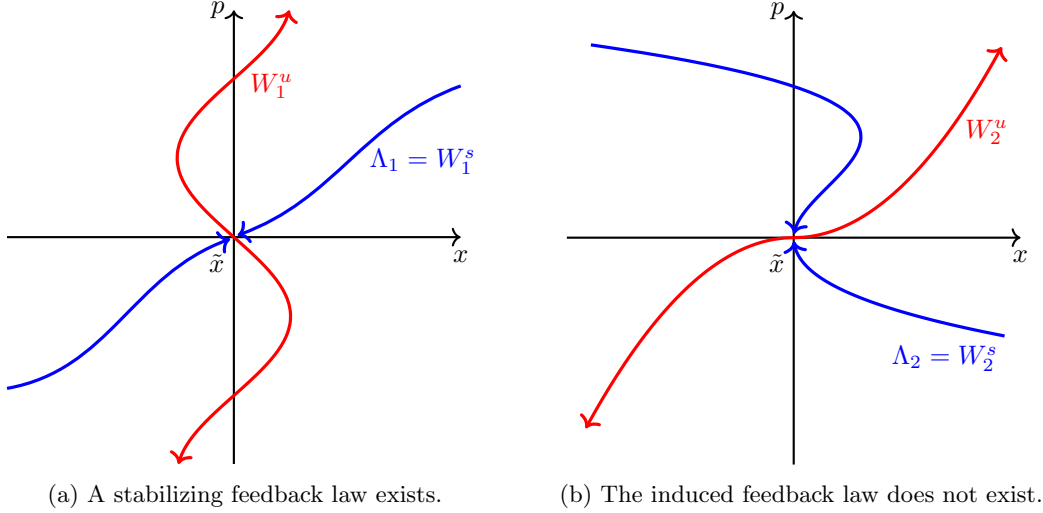
	While the idea of characteristic feedback synthesis and its connection to symplectic geometry is not new \cite{vdschaft_h_infty}, the majority of the focus is on the computational aspect: computing stable manifolds \cite{sakamoto, symplectic_algorithm, av2004, av2003} and computing feedback laws \cite{inverted_pendulum, SAKAMOTO2013568}. 
	A qualitative overview outlining the connection \eqref{eq:insight} (and its converse), as well as extending beyond fixed points, is lacking. The contribution of this work is to explain the geometry of the lifted dynamical system. Three points in particular are:
	\begin{enumerate}
		\item When will $\Lambda$ fail to be a graph? For example, this must be the case if a control system does not satisfy Brockett's necessary condition \cite{brockett1983asymptotic}.
		\item How does this generalize to stabilizing sets more complicated than a fixed-point, e.g., a periodic orbit? If $(\tilde{x},\star)$ is a hyperbolic fixed point for some $\star\in T^*_{\tilde{x}}Q$, then it is known that its stable manifold is Lagrangian. How does this extend to stable manifolds of sets that are not hyperbolic fixed-points? Under characteristic feedback synthesis, more complicated target sets become isotropic manifolds.
		\item This procedure offers an avenue to produce interesting examples in dynamical systems theory whose analysis is aided by results in control theory.
	\end{enumerate}
	
	Generalizing hyperbolic fixed points to more general submanifolds results in Normally Attractive Invariant Manifolds (NAIMs) and Normally Hyperbolic Invariant Manifolds (NHIMs), see, e.g., \cite{eldering2018}. This is too strong an assumption to place on the systems in this work as invariant manifolds in Hamiltonian systems cannot be normally hyperbolic (see Theorem \ref{thm:isotropic_nhim}). The regularity assumption used here will be the following.
	\begin{assumption}\label{ass:naim}
		Let $\varphi_t:M\to M$ be a flow and $N\subset M$ an invariant manifold, i.e. $\varphi_t(N)=N$ for all $t$. Denote its stable set by
		\begin{equation*}
			W^s(N) = \left\{ x\in M : \lim_{t\to\infty} \, \varphi_t(x)\in N \right\}.
		\end{equation*}
		It will be assumed that $W^s(N)$ is:
		\begin{itemize}
			\item[(A1)] (locally) a smooth manifold, and
			\item[(A2)] the tangent dynamics are convergent, i.e. for all tangent vectors $v\in T(W^s(N))$,
			\begin{equation*}
				\lim_{t\to\infty} \, \left( \varphi_t\right)_*v \in TN,
			\end{equation*}
			where $\left( \varphi_t\right)_*$ is the tangent flow induced from $\varphi_t$.
		\end{itemize}
	\end{assumption}
	
	The layout of the paper is the following: Section \ref{sec:pmp} reviews optimal control and includes a statement of the PMP in the context of stabilizing feedback control. Section \ref{sec:symplectic} reviews basic aspects of symplectic geometry, including Hamiltonian systems and Lagrangian/isotropic submanfolds. Section \ref{sec:hyperbolic_fixed} analyzes the connection between feedback control and isotropic submanifolds in Hamiltonian systems.
	Various examples are shown in Section \ref{sec:examples}. Finally, conclusions are in Section \ref{sec:conclusion}.
	
	\section{Optimal Control and Pontryagin's Maximum Principle}\label{sec:pmp}
	\begin{definition}[Control System]
		A \textit{control system} has the form
		\begin{equation*}
			\dot{x} = f(x,u), \quad f:Q\times \mathcal{U}\to TQ,
		\end{equation*}
		where $Q$ is an $n$-dimensional smooth manifold, $\mathcal{U}\subset \mathbb{R}^m$ is a closed set, $x\in Q$ is the state, and $u\in \mathcal{U}$ is the control. It will be assumed that $f$ is smooth.
	\end{definition}
	A fundamental goal of control systems is stabilizing about a fixed point via \textit{continuous} feedback.
	\begin{definition}[{Feedback stabilizable, \cite[\S 10.11]{coron_nonlinear}}]\label{def:feedback_stable}
		Let $\tilde{x}\in Q$ such that $f(\tilde{x},0) = 0$. The control system is \textit{asymptotically feedback stabilizable} by means of a continuous stationary feedback law if there exists a continuous $u\in C^0(Q;\mathcal{U})$ satisfying $u(\tilde{x})=0$ such that
		$\tilde{x}\in Q$ is an asymptotically stable point of the dynamical system
		\begin{equation*}
			\dot{x} = f(x, u(x)).
		\end{equation*}
		Likewise, the system is spectrally stabilizable if the fixed point can be made to be spectrally stabilizable (all eigenvalues of the linearization have strictly negative real part). 
	\end{definition}
	A celebrated necessary condition for a control system to be feedback stabilizable is Brockett's condition \cite{brockett1983asymptotic}; see also \cite[\S 4.5.2]{bloch_nh}, \cite[Theorem~11.1]{coron_nonlinear}, \cite[Theorem~4.1]{liberzon_switching}, and \cite{kvalheim_obstructions}.
	\begin{theorem}[Brockett's condition]
		If the control system $\dot{x}=f(x,u)$ can be (locally) asymptotically stabilized to a fixed point $\tilde{x}\in Q$, then the image of $f$ by every neighborhood of $(\tilde{x},0)\in Q\times \mathcal{U}$ is a neighborhood of $0\in TQ$.
	\end{theorem}
	
	\begin{example}\label{ex:nh_integrator}
		The canonical example of a control system that is controllable but fails Brockett's condition is the nonholonomic integrator:
		\begin{equation*}
			\dot{x} = u, \quad \dot{y} = v, \quad \dot{z} = uy - vx,
		\end{equation*}
		as the vector $(0,0,a)$ does not lie in the image for any $a\ne 0$. This example will be returned to later in \S\ref{sec:nh_integrator}.
	\end{example}
	
	Requiring that the feedback law is spectrally stable results in the following result.
	\begin{proposition}\label{prop:full_rank}
		Suppose that $u$ is a smooth function such that $\dot{x} = f(x,u(x))$ is spectrally stable to $\tilde{x}$. Then
		\begin{equation*}
			\mathrm{rank} \ \left[\begin{array}{c|c}
				\dfrac{\partial f}{\partial x}(\tilde{x}, 0) & \dfrac{\partial f}{\partial u}(\tilde{x}, 0)
			\end{array}\right] = n.
		\end{equation*}
	\end{proposition}
	
	We conclude this section by recalling control Lyapunov functions.
	\begin{definition}[Control Lyapunov function]\label{def:clf}
		A function $V:Q\to\mathbb{R}$ is a \textit{control Lyapunov function} for the control system if
		\begin{itemize}
			\item $V(x)\to +\infty$ as $|x|\to +\infty$,
			\item $V(x)>0$ for all $x\ne \tilde{x}$, and
			\item for all $x\ne \tilde{x}$ there exists $u\in\mathcal{U}$ such that $f(x,u)\cdot\nabla V(x) < 0$.
		\end{itemize}
		Moreover, $V$ has the small control property if for every $\varepsilon>0$ there exists $\eta>0$ such that if $|x|<\eta$, there exists $|u|<\varepsilon$ satisfying $f(x,u)\cdot V(x)<0$.
	\end{definition}
	\begin{theorem}[\cite{ARTSTEIN19831163}]
		If the control system is globally asymptotically stabilizable by means of a continuous stationary feedback law, then it admits a control Lyapunov function satisfying the small control property. The converse is true provided that the control system is affine.
	\end{theorem}
	
	\subsection{Optimal Control}
	Optimal control provides a systematic and constructive way to generate stabilizing feedback laws. For a control system $\dot{x} = f(x,u)$, consider an associated cost function
	\begin{equation}\label{eq:control_cost}
		J(u) := \int_{t_0}^{t_f} \, \ell(x(s), u(s)) \, ds + g\left( x(t_f) \right),
	\end{equation}
	where $\ell$ is the running cost and $g$ the terminal cost. 
	The goal of optimal control is to find the control signal $u^*:[t_0,t_f]\to \mathcal{U}$ that minimizes the cost \eqref{eq:control_cost}, i.e.
	\begin{equation*}
		u^* = \arg\min_u \, J(u).
	\end{equation*}
	Necessary conditions for a control to minimize the cost is the Pontryagin maximum principle. The theorem below is stated for $Q = \mathbb{R}^n$, but its extension to arbitrary manifolds is straightforward.
	
	\begin{definition}[Control Hamiltonian]
		The \textit{control Hamiltonian} for the optimal control problem described above is defined as 
		\begin{gather*}
			H:\mathbb{R}^n \times (\mathbb{R}^n)^* \times\mathbb{R}^+ \times \mathcal{U} \to \mathbb{R}, \\
			H(x, p, \lambda_0, u) = \langle p, f(x,u) \rangle - \lambda_0 \ell(x,u). \notag
		\end{gather*}
	\end{definition}
	
	\begin{theorem}[{Pontryagin maximum principle, \cite[2.2.1]{geo_opt_ctrl}, \cite{Pontryagin1962}}]\label{thm:pmp}
		Let $(x^*,u^*)$ be a controlled trajectory defined over the interval $[t_0,t_f]$ with the control $u^*$ piecewise continuous, i.e.
		\begin{equation*}
			\dot{x}^*(t) = f(x^*(t), u^*(t)).
		\end{equation*}
		If $u^*$ is optimal, then there exist a constant $\lambda_0\geq 0$ and a covector $p:[t_0,t_f]\to (\mathbb{R}^n)^*$ such that the following conditions are satisfied:
		\begin{enumerate}
			\item Nontriviality of the multipliers: $(\lambda_0, p(t))\ne (0,0)$ for all $t\in [t_0,t_f]$.
			\item Adjoint equation: the adjoint variable, $p$, is a solution to the linear differential equation
			\begin{equation*}
				\dot{p}(t) = \lambda_0\frac{\partial \ell}{\partial x}(x^*(t), u^*(t)) - p(t)\frac{\partial f}{\partial x}(x^*(t), u^*(t)).
			\end{equation*}
			\item Maximum condition: for $t\in [t_0,t_f]$ we have
			\begin{equation*}
				H(x^*(t), p(t), \lambda_0, u^*(t)) = \max_{\nu\in\mathcal{U}} \, H(x^*(t), p(t), \lambda_0, \nu).
			\end{equation*}
			\item Transverality condition: at the endpoint of the controlled trajectory, 
			\begin{equation}\label{eq:mom_boundary}
				p(t_f) = \lambda_0 \frac{\partial g}{\partial x}(x(t_f)).
			\end{equation}
		\end{enumerate}
	\end{theorem}
	\begin{remark}\label{rmk:pmp}\mbox{}
		\begin{enumerate}
			\item As the data in the optimal control problem is time-invariant, so is the resulting Hamiltonian and optimality conditions. We make this assumption as we are ultimately interested in constructing stationary (i.e. time-invariant) feedback laws.
			\item If the multiplier $\lambda_0\ne 0$, the trajectory is considered \textit{normal} and \textit{abnormal} otherwise. While abnormal trajectories regularly exist and are important, the systems considered below will be normal and normalized with $\lambda_0=1$.
			\item The control synthesized by Theorem \ref{thm:pmp} is of the form $u = u(t)$ which is not of the form $u(x)$ as is desired for feedback stabilization, recall Definition \ref{def:feedback_stable}. The former is called \textit{open-loop} while the latter is called \textit{closed-loop} control. A modification of Theorem \ref{thm:pmp} to construct closed-loop control laws is given below.
		\end{enumerate}
	\end{remark}
	\subsection{Feedback Stabilization via Optimal Control}
	Suppose that $\tilde{x}$ is a fixed point that is to be stabilized and $f(\tilde{x}, 0)=0$. A \textit{closed-loop} feedback law $u = k(x)$ can be synthesized by imposing the infinite-horizon cost
	\begin{equation*}
		J(u) := \int_{t_0}^\infty \, \ell(x(s), u(s)) \, ds,
	\end{equation*}
	where $\ell(x,u)\geq 0$ and $\ell(\tilde{x},0)=0$. 
	The maximum principle, Theorem \ref{thm:pmp}, states that optimal trajectories satisfy
	\begin{equation}\label{eq:hamiltonian_feedback}
		\begin{split}
			\dot{x} &= f(x, u(x,p)), \\
			\dot{p} &= \frac{\partial \ell}{\partial x}(x, \mu(x,p)) - p\frac{\partial f}{\partial x}(x, \mu(x,p)), \\
			\mu(x,p) &= \arg\max_{\nu \in \mathcal{U}} H(p, x, \nu),
		\end{split}
	\end{equation}
	where $\lambda_0$ has been omitted by Remark \ref{rmk:pmp}. However, the control law depends on the co-state as well as the state. To remove the dependency on the co-state, denote the state trajectory of \eqref{eq:hamiltonian_feedback} by $x(t;x_0,p_0)$ where $x_0$ and $p_0$ denote the initial values. Then, define
	\begin{equation}\label{eq:stable_momenta}
		\rho(x_0) \text{~s.t.~} \lim_{t\to\infty} \, x(t; x_0, \rho(x_0)) = x_0.
	\end{equation}
	Assuming that there exists a unique solution for $\rho(x_0)$, a stabilizing feedback law is given by
	\begin{equation*}
		k(x) = \mu(x, \rho(x)).
	\end{equation*}
	While Brockett's condition provides a \textit{topological} obstruction, we will explore the \textit{dynamical} obstructions to the procedure \eqref{eq:hamiltonian_feedback} and \eqref{eq:stable_momenta}. As the dynamical system \eqref{eq:hamiltonian_feedback} is Hamiltonian, it is natural to translate this problem to symplectic geometry. Moreover, the condition \eqref{eq:stable_momenta} translates to an invariant Lagrangian submanifold.
	
	\begin{remark}
		The function \eqref{eq:stable_momenta} is related to control Lyapunov functions (recall Definition \ref{def:clf}). As will be clear in the following section, $H(x, \rho(x)) = 0$ for all $x$. Let $S$ be a function such that $-\nabla S = \rho$. Then $S$ is a control Lyapunov function. Indeed,
		\begin{equation*}
			\langle \nabla S(x), f(x,u)\rangle = -\langle \rho(x), f(x,u) \rangle = -\ell\left(x, \mu(x, \rho(x))\right) \leq 0,
		\end{equation*}
		with equality when $x=\tilde{x}$.
	\end{remark}
	
	
	\section{Symplectic Geometry}\label{sec:symplectic}
	This section provides a quick review of symplectic geometry, see \cite{berndt2001} for a concise overview. Throughout, it will be assumed that all manifolds are finite-dimensional, and all objects are smooth.
	
	\begin{definition}[Symplectic Manifold]
		A \textit{symplectic manifold} is a pair $(M,\omega)$ where $M$ is a (smooth, finite-dimensional) manifold and $\omega\in\Omega^2(M)$ is a closed, non-degenerate two-form.
	\end{definition}
	
	The central example of symplectic manifolds and those that will be studied in this work are the contangent bundles. Let $Q$ be a $n$-dimensional manifold and $\pi_Q:T^*Q\to Q$ be its cotangent bundle. The Liouville form $\vartheta\in\Omega^1(T^*Q)$ is given via
	\begin{equation*}
		\vartheta_{(x,p)}(\dot{x},\dot{p}) = p\left( \dot{x}\right) ,
	\end{equation*}
	where $\vartheta = p_i dx^i$ in local coordinates; repeated indices implying a summation will be used throughout.
	The Liouville form induces a symplectic structure on $T^*Q$ via
	\begin{equation}\label{eq:cotangent_hamiltonian}
		\omega = -d\vartheta = dx^i\wedge dp_i.
	\end{equation}
	Symplectic manifolds are the natural structures for Hamiltonian systems.
	\begin{definition}[Hamiltonian System]
		Let $(M,\omega)$ be a symplectic manifold and $H:M\to \mathbb{R}$ a smooth function. Then the unique vector field $X_H\in\mathfrak{X}(M)$ is called a Hamiltonian vector field with Hamiltonian function $H$ if
		\begin{equation}\label{eq:ham_eq}
			i_{X_H}\omega = dH,
		\end{equation}
		where $i_{X}$ is the inner multiplication of a vector field and a differential form. The triple $(M,\omega,X_H)$ is called a \textit{Hamiltonian system}.
	\end{definition}
	In local coordinates with the symplectic form given by \eqref{eq:cotangent_hamiltonian}, Hamilton's equations \eqref{eq:ham_eq} take the familiar form:
	\begin{equation}\label{eq:local_hamilton_eq}
		\dot{x}^i = \frac{\partial H}{\partial p_i}, \quad \dot{p}_i = -\frac{\partial H}{\partial x^i}.
	\end{equation}
	In the remainder of this work, $M$ will refer to an arbitrary symplectic manifold while $T^*Q$ will carry the canonical form \eqref{eq:cotangent_hamiltonian}.
	
	It can be seen that Pontraygin's maximum principle, Theorem \ref{thm:pmp}, induces a Hamiltonian system on $T^*Q$ with Hamiltonian
	\begin{gather*}
		H:T^*Q\to\mathbb{R} \\
		H(x, p) = \max_{u\in\mathcal{U}} \, \langle p, f(x,u)\rangle - \ell(x, u).
	\end{gather*}
	Hamilton's equations \eqref{eq:ham_eq} with the symplectic form \eqref{eq:cotangent_hamiltonian} is equivalent to the adjoint equations \eqref{eq:hamiltonian_feedback}.
	
	Two fundamental features of Hamiltonian systems are energy and volume conservation.
	\begin{theorem}\label{thm:symplectic_preserve}
		Let $(M,\omega, X_H)$ be a Hamiltonian system with Hamiltonian function $H$. Then the flow $\varphi_t$ induced by the dynamics $\dot{z}=X_H(z)$ satisfies:
		\begin{description}
			\item[Energy Conservation:] $\varphi_t^*H=H$.
			\item[Symplecticity:] $\varphi_t^*\omega = \omega$.
		\end{description}
	\end{theorem}
	
	\subsection{Isotropic Submanifolds}
	While the adjoint equations in optimal control \eqref{eq:hamiltonian_feedback} are Hamiltonian, the stability condition \eqref{eq:stable_momenta} is encoded by an invariant Lagrangian submanifold, which is a maximal isotropic submanifold.
	\begin{definition}[Isotropic submanifolds]
		An immersed submanifold $\iota:N\hookrightarrow (M,\omega)$ is \textit{isotropic} if $\iota^*\omega = 0$. An isotropic submanifold is Lagrangian if $\dim N = 1/2\dim M$.
	\end{definition}
	
	\begin{proposition}[{\cite[5.3.32]{abraham_marsden}}]\label{prop:inv_lagrangian}
		Let $\Lambda\subset (M,\omega)$ be Lagrangian and $H\in C^\infty(M)$.
		\begin{enumerate}
			\item If $\varphi_t$ is the flow of $X_H$, then $\varphi_t(\Lambda)\subset M$ is Lagrangian.
			\item If $H$ is constant on $\Lambda$, then $\Lambda$ is invariant under the flow of $X_H$.
		\end{enumerate}
	\end{proposition}
	
	\begin{example}
		Let $\mathbb{R}^{2n}$ be endowed with the symplectic form: 
		\begin{equation*}
			\omega\left( (x,u), (y,v) \right) = v^\top x - u^\top y.
		\end{equation*}
		A subspace of the form
		\begin{equation}\label{eq:LQR_lagrangian}
			\Lambda = \left\{ (x, Px) : x\in\mathbb{R}^n\right\},
		\end{equation}
		is Lagrangian if and only if $P=P^\top$. Let $H$ be a quadratic Hamiltonian,
		\begin{equation*}
			H(x,p) = \frac{1}{2} x^\top Q x + p^\top A x + \frac{1}{2}p^\top R p.
		\end{equation*}
		The condition that $\Lambda \in H^{-1}(0)$ is equivalent to the continuous-time algebraic Riccati equation:
		\begin{equation}\label{eq:CARE}
			Q + A^\top P + PA + PRP = 0.
		\end{equation}
		If $P$ solves \eqref{eq:CARE}, then the subspace \eqref{eq:LQR_lagrangian} is invariant under the dynamics
		\begin{equation*}
			\dot{x} = Ax + Rp, \quad \dot{p} = -Qx - A^\top p.
		\end{equation*}
	\end{example}
	
	A characterization of Lagrangian/isotropic submanifolds is given by the following proposition, cf. \cite[5.3.15]{abraham_marsden}.
	\begin{proposition}\label{prop:closed_Lagrangian}
		Let $M = T^*Q$ be symplectic with the canonical symplectic form and let $S\subset Q$ be a submanifold. Then:
		\begin{enumerate}
			\item $Z_Q \subset T^*Q$ is Lagrangian, where $Z_Q$ is the zero section,
			\begin{equation*}
				Z_Q = \left\{ (x, 0) : x\in Q\right\} \subset T^*Q.
			\end{equation*}
			\item $T_x^*Q \subset T^*Q$ is Lagrangian for a fixed $x\in Q$.
			\item Let $\alpha$ be a 1-form on $Q$ and $\Gamma_\alpha \subset T^*Q$ its graph. Then $\Gamma_\alpha$ is Lagrangian if and only if $\alpha$ is closed.
		\end{enumerate}
		Additionally, submanifolds of Lagrangian manifolds are isotropic. In particular,
		\begin{enumerate}\setcounter{enumi}{3}
			\item $Z_S\subset T^*Q$ is isotropic, where
			\begin{equation*}
				Z_S = \left\{ (x, 0) : x\in S\right\} \subset T^*Q.
			\end{equation*}
			\item Subspaces of a cotangent fiber, $V\subset T_x^*Q$, are isotropic.
			\item If $\dim N = 1$, then $N$ is isotropic.
		\end{enumerate}
	\end{proposition}
	
	\begin{corollary}
		A Lagrangian submanifold $\Lambda\subset T^*Q$ is locally the image of a differential of a function on $Q$ if and only if
		\begin{enumerate}
			\item $\Lambda$ is isotropic, and
			\item $\Lambda$ is transverse to the fibers of $T^*Q$.
		\end{enumerate}
	\end{corollary}
	
	\begin{definition}[Generating Function]
		Let $\Lambda\subset T^*Q$ be a Lagrangian submanifold. A function $S:Q\to\mathbb{R}$ is a (local) \textit{generating function} if $\Lambda$ is (locally) the graph of $dS$.
	\end{definition}
	It is not uncommon for Lagrangian submanifolds to fail to have a global generating function; this is tied to the well-known fact that the Hamilton-Jacobi(-Bellman) equations do not generally admit smooth solutions \cite{viscosity_fields}.
	Shocks in solutions to the HJB correspond to locations where $\Lambda$ fails to be transverse to the fibers of $T^*Q$ and are called \textit{caustics}. An initial study on caustics are given in \cite{guckenheimer}. See also \cite[Appendix~12]{arnold}. 
	\begin{definition}[Caustic]
		The \textit{caustic} of a Lagrangian manifold $\Lambda\subset T^*Q$ is the critical value set of $\pi_Q:T^*Q\to Q$. This set is denoted by $\Sigma(\Lambda)\subset Q$.
	\end{definition}
	\begin{example}
		The set $\Sigma(\Lambda)$ can be fairly complicated. Consider the Hamiltonian on $\mathbb{R}^2$:
		\begin{equation*}
			H(x, p) = p\left( x - e^{-p}\sin p \right).
		\end{equation*}
		An invariant Lagrangian submanifold is
		\begin{equation}\label{eq:wiggly_Lagrangian}
			\Lambda = \left\{ (e^{-p}\sin p, p) : p\in\mathbb{R} \right\}.
		\end{equation}
		The caustic set is (see Fig. \ref{fig:wiggly_Lagrangian})
		\begin{equation*}
			\Sigma(\Lambda) = \left\{ \pm \frac{1}{\sqrt{2}} \exp\left( \frac{1}{4}(4n\pi + \pi) \right) : n\in\mathbb{Z} \right\}.
		\end{equation*}
		\begin{figure}
			\centering
			\begin{tikzpicture}
				\draw[thick, ->] (-4,0) -- (4,0);
				\draw[thick, ->] (0,-1) -- (0,5);
				\draw[very thick, blue, domain=-0.5:5, smooth] plot ({4*exp(-0.5*\x)*sin(deg(1.75*\x))}, \x); 
				\node[below] at (4,0) {$x$};
				\node[left] at (0,5) {$p$};
				\node[blue] at (3,1) {$\Lambda$};
				\draw[dashed] (-1.0835,2.5338) -- (-1.0835,0);
				\draw[dashed] (0.4416,4.3290) -- (0.4416,0);
				\draw[dashed] (2.6585,0.7386) -- (2.6585,0);
				\draw[red, fill] (-1.0835,2.5338) circle [radius=0.1];
				\draw[red, fill] (0.4416,4.3290) circle [radius=0.1];
				\draw[red, fill] (2.6585,0.7386) circle [radius=0.1];
				\draw[fill] (-1.0835,0) circle [radius=0.1];
				\draw[fill] (0.4416,0) circle [radius=0.1];
				\draw[fill] (2.6585,0) circle [radius=0.1];
			\end{tikzpicture}
			\caption{Three points in the caustic set $\Sigma(\Lambda)$ for the Lagrangian submanifold \eqref{eq:wiggly_Lagrangian}.}
			\label{fig:wiggly_Lagrangian}
		\end{figure}
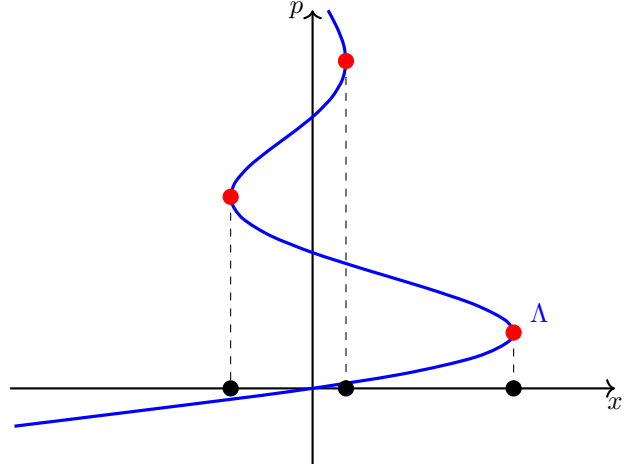
	\end{example}
	
	\section{Stabilizing Manifolds}\label{sec:hyperbolic_fixed}
	We return to the problem of feedback synthesis. Let $S\subset Q$ be a submanifold that we wish to stabilize. In particular, assume that
	\begin{enumerate}
		\item $S$ is invariant under the uncontrolled flow,
		\begin{equation*}
			f(x, 0) \in T_xS, \quad x\in S.
		\end{equation*}
		\item and the running cost is minimized on $S$,
		\begin{equation*}
			\ell(x, u) \geq 0, \quad x\in S \implies \ell(x,0) = 0.
		\end{equation*}
	\end{enumerate}
	Does the procedure \eqref{eq:hamiltonian_feedback} and \eqref{eq:stable_momenta} produce a feedback that makes the manifold $S$ asymptotically stable?
	As a first step to answering this question, we state the following symplectic result which is a strict generalization of fixed points (see \cite[\S 6.4]{katok_hasselblatt} and \cite[Lemma~1]{vdschaft_h_infty}). Recall that for a given $S\subset Q$, the zero section $N=Z_S\subset T^*Q$ is isotropic.
	\begin{theorem}\label{thm:att_man_iso}
		Let $N \subset M$ be an isotropic submanifold with stable manifold,
		\begin{equation*}
			W^s(N) = \left\{ x\in M : \lim_{t\to\infty} \varphi_t(x) \in N\right\},
		\end{equation*}
		which satisfies Assumption \ref{ass:naim}.
		Then $W^s(N)\subset M$ is also an isotropic manifold.
	\end{theorem}
	\begin{proof}
		Let $v,w\in T_xW^s(N)$. Then 
		\begin{equation*}
			\omega(v,w) = \omega\left( (\varphi_t)_*v, (\varphi_t)_*w\right) = \lim_{t\to\infty} \omega\left( (\varphi_t)_*v, (\varphi_t)_*w\right),
		\end{equation*}
		as the flow is symplectic, $\varphi_t^*\omega=\omega$ (recall Theorem \ref{thm:symplectic_preserve}). The right hand side vanishes as $N$ is assumed to by isotropic, i.e. $\omega(\mu,\nu)=0$ for all $\mu,\nu\in TN$.
	\end{proof}
	
	An important consequence of the above is that 
	\begin{equation}\label{eq:stable_manifold_dimension}
		\dim W^s(Z_S) \leq \dim Q.
	\end{equation}
	Much of the theory on stable manifolds requires some notion of hyperbolicity. Unfortunately, this cannot be the case for isotropic manifolds. See \cite[Proposition~10]{Lomeli_2008} for a similar result concerning diffeomorphisms.
	\begin{theorem}\label{thm:isotropic_nhim}
		Let $N\subset M$ be an isotropic manifold with $\dim N \ne 0$. Then $N$ is not normally hyperbolic under any Hamiltonian flow.
	\end{theorem}
	\begin{proof}
		Suppose that $N$ is normally hyperbolic and let $x\in N$. Then we have the hyperbolic splitting
		\begin{equation*}
			T_xM = T_xN \oplus E^s_x \oplus E^u_x.
		\end{equation*}
		As NHIMs satisfy Assumption \ref{ass:naim}, we know that the stable (and unstable) manifold are isotropic. Therefore,
		\begin{equation*}
			\dim T_xN + \dim E^s_x \leq n ,\quad \dim T_xN + \dim E^u_x \leq n.
		\end{equation*}
		As $\dim T_xN \geq 1$, it follows that 
		\begin{equation*}
			\begin{split}
				2n &= \dim T_x M = \dim T_xN + \dim E_x^s + \dim E_x^u \\
				&< 2\dim T_xN + \dim E_x^s + \dim E_x^u \leq 2n,
			\end{split}
		\end{equation*}
		which is a contradiction.
	\end{proof}
	As hyperbolicity cannot occur in isotropic manifolds (unless $N$ is a fixed point), the closest that can happen is equality in \eqref{eq:stable_manifold_dimension} along with Assumption \ref{ass:naim}.
	\begin{definition}[Symplectic saddle]
		Let $N\subset (M,\omega)$ be an isotropic submanifold and $X_H$ a Hamiltonian vector field. The pair $(N, X_H)$ is a \textit{symplectic saddle} if Assumption \ref{ass:naim} holds and the manifolds $W^s(N)$ and $W^u(N)$ are Lagrangian.
	\end{definition}
	
	\begin{problem}[Symplectic saddle problem]\label{quest:LAP}
		Let $N\subset M$ be an isotropic submanifold and $H:M\to\mathbb{R}$ a Hamiltonian. Is $(N,X_H)$ a symplectic saddle?
	\end{problem}
	Hyperbolic fixed points provide a trivial solution to the above problem. When $N$ is a periodic orbit (recall that all one-dimensional submanifolds are isotropic), the above problem is solved if $N$ is hyperbolic when restricted to an energy level. This is straightforward to determine by checking its Floquet multipliers and ensuring that only two are degenerate. Although it is not stated in this language, if $N$ is an orbit of a group action, then being a symplectic saddle is related to the notion of \textit{relative equilibria}. For a nice presentation on this topic, see \cite{marsden_lectures}. However, this case is quite restrictive and will not be pursued here.
	
	Applying this back to feedback synthesis, we conclude.
	\begin{theorem}\label{thm:main_result}
		Consider an optimal control problem satisfying the assumptions at the beginning of this section, namely that $S$ is invariant and minimizes the cost. Let $H:T^*Q\to\mathbb{R}$ be the induced Hamiltonian and assume that the map
		\begin{gather*}
			\mu:T^*Q\to\mathcal{U} \\
			\mu(x,p) = \arg\max_{u\in\mathcal{U}} \langle p, f(x,u)\rangle - \ell(x,u),
		\end{gather*}
		is smooth. If
		\begin{enumerate}
			\item $(Z_S, X_H)$ is a symplectic saddle, and
			\item $\pi_Q:W^s(Z_S)\to Q$ is a diffeomorphism in a tubular neighborhood of $S$,
		\end{enumerate}
		then there exists (in a tubular neighborhood) a smooth feedback law that stabilizes $S$.
	\end{theorem}
	\subsection{Fixed points}
	This section examines the special case where $S = \{\tilde{x}\}$ as hyperbolic fixed points are always a symplectic saddle. We begin with the classical result that states that hyperbolic fixed points in Hamiltonian systems have an equal number of eigenvalues with positive and negative real part.
	\begin{proposition}[{\cite[14.3.5]{wiggins}}]\label{prop:ham_ews}
		Let $A$ be an infinitesimally symplectic matrix. If $\lambda\in\mathbb{C}$ is an eigenvalue of $A$, so are $-\lambda$, $\bar{\lambda}$, and $-\bar{\lambda}$. Moreover, if $0$ is an eigenvalue, then it has even multiplicity.
	\end{proposition}
	
	\begin{corollary}
		Suppose that $A$ is a hyperbolic matrix. Then $\dim E^s = \dim E^u = n$, where $E^s$ and $E^u$ are the stable and unstable subspaces, respectively. In particular, hyperbolic fixed points are a symplectic saddle.
	\end{corollary}
	If $\tilde{x}\in Q$ is a fixed point in the control system, it is not automatic that $(\tilde{x},0)\in T^*Q$ is hyperbolic. 
	\begin{proposition}
		Suppose that $\dot{x}=f(x,u)$ has a spectral-stable feedback generated by the cost $\ell$. Then $(\tilde{x}, 0)$ is a hyperbolic fixed point of the induced Hamiltonian system.
	\end{proposition}
	\begin{proof}
		The control at the fixed point is
		\begin{equation*}
			\begin{split}
				u(\tilde{x}, 0) &= \arg\max_u \, \langle 0, f(\tilde{x},u)\rangle - \ell(\tilde{x}, u) \\
				&= \arg\max_u - \ell(\tilde{x}, u).
			\end{split}
		\end{equation*}
		As $\ell$ is positive-definite in $u$, it follows that $u(\tilde{x}, 0)=0$. Therefore $(\tilde{x},0)$ is a fixed point of $X_H$. As the tangent map has $n$ eigenvalues with negative real-part (via spectral-stability), there exist $n$ eigenvalues with positive real-part by Proposition \ref{prop:ham_ews}. Therefore, it is a hyperbolic fixed point.
	\end{proof}
	
	Let $(\tilde{x},0)$ be a hyperbolic fixed point of a Hamiltonian system on $T^*Q$. As $W^s(\tilde{x},0)$ is Lagrangian, there exists a generating function (in a neighborhood of $\tilde{x}$) as long as $\tilde{x}\not\in \Sigma\left( W^s(\tilde{x},0)\right)$.
	This can be tested by the stable manifold theorem (e.g., \cite[Theorem~3.2.1]{wiggins}) as the tangent space is explicitly known and $T_{(\tilde{x},0)}W^s(\tilde{x},0) = E^s$.
	This leads immediately to the following (which can be viewed as the special case of Theorem \ref{thm:main_result} applied to fixed points).
	\begin{theorem}\label{thm:generating_fixed}
		Let $(\tilde{x}, 0)$ be a hyperbolic fixed point to a Hamiltonian system on $T^*Q$. Let $v^s_1, \ldots, v^s_n \subset T_{(\tilde{x}, 0)}T^*Q$ be $n$ linearly independent stable eigenvectors for $(\tilde{x}, 0)$. If the collection
		\begin{equation*}
			\left\{ T\pi_Q(v_1^s), \ldots, T\pi_MQ(v_n^s) \right\} \subset T_{\tilde{x}}Q,
		\end{equation*}
		is linearly independent, 	
		then there (locally) exists a function $S:Q\to\mathbb{R}$ such that
		\begin{equation*}
			W^s(\tilde{x}, 0) = \Gamma_{dS}.
		\end{equation*}
	\end{theorem}
	The above theorem can be expressed in local coordinates on $T^*Q$. The linearization of the Hamiltonian vector field \eqref{eq:local_hamilton_eq} at a fixed point is
	\begin{equation*}
		Z = \left[\begin{array}{cc}
			\dfrac{\partial^2 H}{\partial x\partial p} & \dfrac{\partial^2H}{\partial p^2} \\[2ex]
			-\dfrac{\partial^2H}{\partial x^2} & -\dfrac{\partial^2H}{\partial x \partial p}
		\end{array}\right].
	\end{equation*}
	Suppose that this matrix is hyperbolic with stable eigenvectors $v_1^s,\ldots, v_n^s$. Construct the $2n\times n$ matrix
	\begin{equation}\label{eq:stable_directions}
		S = \begin{bmatrix} U \\ V \end{bmatrix} = \left[ \begin{array}{c|c|c}
			v_1^s & \cdots & v_n^s
		\end{array}\right].
	\end{equation} 
	Then, Theorem \ref{thm:generating_fixed} states that there locally exists a generating function if the $n\times n$ matrix $U$ is invertible.
	Results from infinite-horizon LQR provide sufficient conditions for the matrix $U$ to be invertible, cf. Theorem 2 in \cite{LQR_existence} and Theorem 6.1 in \cite{liberzon_oc}.
	\begin{theorem}\label{thm:LQR_U}
		Consider the matrix
		\begin{equation*}
			Z = \begin{bmatrix} A & BR^{-1}B^\top \\
				C^\top C & -A^\top \end{bmatrix},
		\end{equation*}
		where $R=R^\top >0$ is symmetric and positive-definite, $(A,B)$ is a controllable pair, and $(A,C)$ is an observable pair. Then the matrix $U$ in \eqref{eq:stable_directions} is invertible.
	\end{theorem}
	
	Mechanical systems provide an explicit class of examples that satisfy the above conditions.
	\begin{theorem}
		Let $H:T^*Q\to \mathbb{R}$ be a mechanical Hamiltonian, i.e.
		\begin{equation*}
			H(x,p) = \frac{1}{2}g^{ij}(x)p_ip_j + V(x),
		\end{equation*}
		such that $\tilde{x}$ is a non-degenerate maximum of $V$. Then $(\tilde{x},0)$ is a hyperbolic fixed point and $W^s(\tilde{x},0)$ is Lagrangian with a local generating function.
	\end{theorem}
	\begin{proof}
		The tangent of the Hamiltonian vector field at $(x^*,0)$ is
		\begin{equation*}
			Z = \left. \left[ \begin{array}{c:c}
				\dfrac{\partial g}{\partial x^i}p_j & \left(g^{ij}\right) \\[1ex] \hdashline \noalign{\vskip 2pt}
				-\dfrac{\partial^2 V}{\partial x^i\partial x^j} & -\dfrac{\partial g}{\partial x^j}p_i
			\end{array}\right] \right|_{x=x^*, p=0} =: \left[ \begin{array}{c:c}
				0 & N \\ \hdashline
				Q & 0
			\end{array}\right],
		\end{equation*}
		where each block is an $n\times n$ matrix. As both $N$ and $Q$ are positive-definite, there exists invertible matrices $B$ and $C$ such that
		\begin{equation*}
			N = BB^\top, \quad Q = C^\top C.
		\end{equation*}
		The result follows from Theorem \ref{thm:LQR_U}.
	\end{proof}
	Theorem \ref{thm:generating_fixed} also allows for a \textit{dynamical} interpretation of Brockett's condition.
	\begin{proposition}
		Suppose that Brockett's condition fails for the control system $\dot{x}=f(x,u)$. If $(\tilde{x},0)$ is a fixed point of the induced Hamiltonian system, then the fixed point lies on a caustic, i.e. $\tilde{x} \in \Sigma \left( W^s(\tilde{x},0) \right)$.
	\end{proposition}
	\begin{example}[The nonlinear pendulum]\label{ex:pendulum}
		Let $Q = \mathbb{S}^1$ be the unit circle and $M = T^*\mathbb{S}^1$ be its cotangent bundle with local coordinates $(\theta, p)$ and symplectic form $\omega = d\theta\wedge dp$. Define the Hamiltonian function
		\begin{equation}\label{eq:pendulum_hamiltonian}
			H:T^*\mathbb{S}^1\to\mathbb{R}, \quad (\theta,p) \mapsto \frac{1}{2}p^2 + \cos\theta.
		\end{equation}
		The resulting Hamiltonian vector field is
		\begin{equation*}
			X_H = p\frac{\partial}{\partial \theta} + \sin\theta\frac{\partial}{\partial p}.
		\end{equation*}
		The origin is a hyperbolic fixed point with eigenvalues $\lambda = \pm 1$. The stable/unstable manifolds are Lagrangian submanfolds:
		\begin{equation*}
			\begin{split}
				W^s(0,0) &= \left\{ (\theta, \sqrt{2-2\cos\theta}) : -\pi<\theta\leq 0 \right\} \cup \left\{ (\theta, -\sqrt{2-2\cos\theta}) : 0\leq\theta\leq \pi \right\},\\
				W^u(0,0) &= \left\{ (\theta, -\sqrt{2-2\cos\theta}) : -\pi<\theta\leq 0 \right\} \cup \left\{ (\theta, \sqrt{2-2\cos\theta}) : 0\leq\theta\leq \pi \right\},
			\end{split}
		\end{equation*}
		see Figure \ref{fig:pendulum_manifold}.
		As $W^s$ is a Lagrangian manifold, by Proposition \ref{prop:closed_Lagrangian}, there exists (locally) a generating function. This can be calculated and is
		\begin{equation}\label{eq:gen_fcn_pendulum}
			S(\theta) = \begin{cases}
				4 & \theta =0, \\
				-2\sqrt{2-2\cos\theta}\cot\dfrac{\theta}{2}, & -2\pi<\theta<0, \\
				2\sqrt{2-2\cos\theta}\cot\dfrac{\theta}{2}, & 0<\theta < 2\pi.
			\end{cases}
		\end{equation}
		Notice that \eqref{eq:gen_fcn_pendulum} is a solution to the Hamilton-Jacobi equation
		corresponding to \eqref{eq:pendulum_hamiltonian}:
		\begin{equation*}
			\frac{1}{2}\left( \frac{\partial S}{\partial \theta} \right)^2 + \cos\theta = 1.
		\end{equation*}
		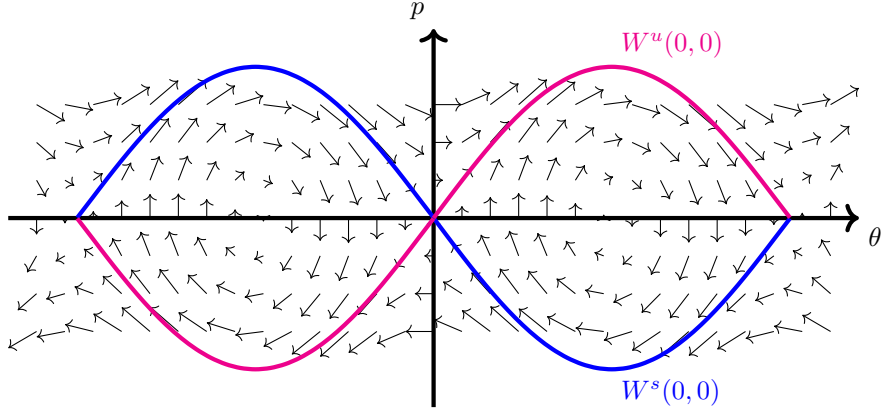
\begin{figure}
			\centering
			\begin{tikzpicture}[xscale=0.75]
				\draw[->, ultra thick] (-7.5, 0) -- (7.5, 0);
				\draw[->, ultra thick] (0, -2.5) -- (0, 2.5);
				\node[below right] at (7.5, 0) {$\theta$};
				\node[above left] at (0, 2.5) {$p$};
				\foreach \x in {-7,-6.5,...,7}{
					\foreach \y in {-1.5,-1,...,1.5}{
						\pgfmathsetmacro{\vx}{\y}
						\pgfmathsetmacro{\vy}{sin(deg(\x))}
						\draw[->] (\x,\y) -- (\x+\vx/3,\y+\vy/3);
					}
				}
				\draw[ultra thick, blue, domain=-2*pi:0, smooth] plot (\x, {sqrt(2-2*cos(deg(\x)))});
				\draw[ultra thick, blue, domain=0:2*pi, smooth] plot (\x, {-sqrt(2-2*cos(deg(\x)))});
				\draw[ultra thick, magenta, domain=-2*pi:0, smooth] plot (\x, {-sqrt(2-2*cos(deg(\x)))});
				\draw[ultra thick, magenta, domain=0:2*pi, smooth] plot (\x, {sqrt(2-2*cos(deg(\x)))});
				\node[above right, magenta] at (pi,2) {$W^u(0,0)$};
				\node[below right, blue] at (pi,-2) {$W^s(0,0)$};
			\end{tikzpicture}
			\caption{The stable and unstable manifolds for the hyperbolic fixed point in the nonlinear pendulum from Example \ref{ex:pendulum}.}
			\label{fig:pendulum_manifold}
		\end{figure}
	\end{example}
	\section{Examples}\label{sec:examples}
	We provide numerous examples demonstrating the application and failures of Theorems \ref{thm:main_result} and \ref{thm:generating_fixed}. The first two examples, \S\ref{sec:uncontrol} and \S\ref{sec:local_stable}, are purely mathematical and are chosen such that the induced stable manifold has desired properties. In particular, these examples are not strictly optimization problems as the running cost is not positive-definite. The next example, \S\ref{sec:nh_integrator}, is Brockett's nonholonomic integrator which is known to not admit a continuous feedback law stabilizing the origin. The last example, \S\ref{sec:periodic_orbit}, stabilizes a periodic orbit in a planar system. All numerical results for the last two examples used the DifferentialEquations.jl Julia package \cite{rackauckas2017differentialequations}.
	\subsection{Uncontrollable system}\label{sec:uncontrol}
	If a control system is uncontrollable, stabilizing feedback laws obviously cannot exist. This should be apparent from its stable manifold. 
	Consider the control system and cost
	\begin{equation*}
		\dot{x} = x - u^2, \quad \ell = x^2 + u^3.
	\end{equation*}
	Applying the maximum principle yields the Hamiltonian
	\begin{equation}\label{eq:hamiltonian_quadratic}
		H = -\frac{4}{27}p^3 + px - x^2,
	\end{equation}
	with the control law $u = -2/3p$. The origin ends up being a hyperbolic homoclinic point, see Fig. \ref{fig:quadratic_loop}. 
	
	The caustic set is $\Sigma(W^s) = \{0,1\}$ which contains the fixed point obstructing smooth feedback stabilization. Moreover, the stable manifold is only in the first and fourth quadrants signifying that the underlying control system is uncontrollable.
	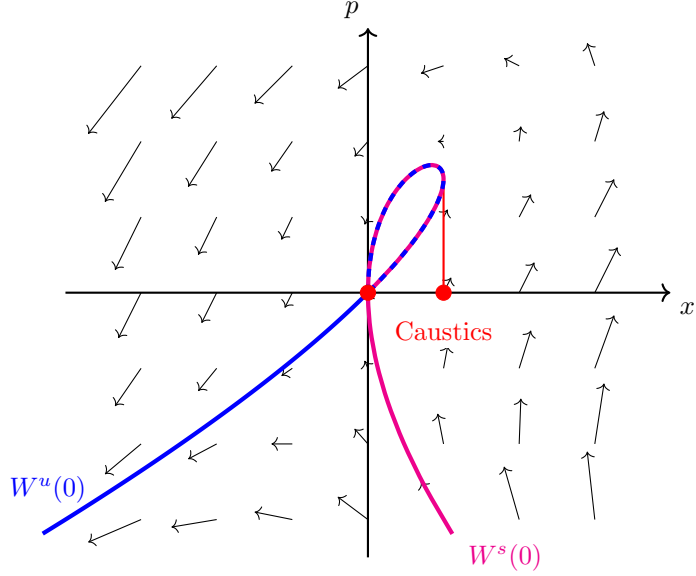
\begin{figure}
		\centering
		\begin{tikzpicture}
			\draw[->, thick] (-4, 0) -- (4, 0);
			\draw[->, thick] (0,-3.5) -- (0,3.5);
			\node[below right] at (4, 0) {$x$};
			\node[above left] at (0, 3.5) {$p$};
			\foreach \x in {-3,-2,...,3}{
				\foreach \y in {-3,-2,...,3}{
					\pgfmathsetmacro{\vx}{\x-4/9*\y*\y}
					\pgfmathsetmacro{\vy}{2*\x-\y}
					\draw[->] (\x,\y) -- (\x+\vx/10,\y+\vy/10);
				}
			}
			\draw[thick, red] (1,3/2) -- (1,0);
			\draw[red, fill] (1,0) circle [radius=0.1];
			\draw[ultra thick, magenta, domain=-0.35:1.0, smooth] plot ({6.75*\x*\x*(1-\x)},{6.75*\x*(1-\x)});
			\draw[ultra thick, blue, domain=0.0:1.0, smooth, dashed] plot ({6.75*\x*\x*(1-\x)},{6.75*\x*(1-\x)});
			\draw[ultra thick, blue, domain=1.0:1.35, smooth] plot ({6.75*\x*\x*(1-\x)},{6.75*\x*(1-\x)});
			\draw[red, fill] (0,0) circle [radius=0.1];
			\node[magenta, below right] at (1.2, -3.2) {$W^s(0)$};
			\node[blue, above left] at (-3.6,-2.9) {$W^u(0)$};
			\node[red, fill=white] at (1,-0.5) {Caustics};
		\end{tikzpicture}
		\caption{The stable manifolds for the origin in the Hamiltonian system \eqref{eq:hamiltonian_quadratic}. As the origin is a homoclinic point, the transition between the stable and unstable manifold in the first quadrant is arbitrary as they coincide there.}
		\label{fig:quadratic_loop}
	\end{figure}
	\subsection{Locally stabilizable system}\label{sec:local_stable}
	It can be the case that a control system admits a local feedback law while a global feedback is impossible. This can be seen from the stable manifold having ``hysteresis.''
	
	Consider the control system and cost
	\begin{equation}\label{eq:cubic_hysteresis_dynamics}
		\dot{x} = x - 4u^3 + 6u, \quad \ell = 3u^2(1-u^2).
	\end{equation}
	Applying the maximum principle yields the Hamiltonian
	\begin{equation*}
		H = px - p^4 + 3p^2,
	\end{equation*}
	with the control law $u = p$. The origin is a hyperbolic fixed point with corresponding manifolds:
	\begin{equation*}
		\Lambda = W^s(0,0) = \left\{ x = p^3-3p \right\}, \quad W^u(0,0) = \left\{ p=0 \right\},
	\end{equation*}
	see Fig. \ref{fig:cubic_hysteresis}.
	
	When $-2<x<2$, the projection $\pi:\Lambda\to \mathbb{R}$ is 3-to-1; each choice of $x$ in that interval yields three possible control laws. A global continuous section is impossible. This yields the option:
	\begin{itemize}
		\item A continuous feedback law (given in blue in Fig. \ref{fig:cubic_hysteresis_decomposition}). However, this law cannot be extended beyond the interval $[-2,2]$. This is because the caustic set is $\Sigma(\Lambda) = \{-2,2\}$.
		\item A discontinuous feedback law (given in magenta in Fig. \ref{fig:cubic_hysteresis_decomposition}) which is defined for on the entire line $\mathbb{R}$. In this case, the resulting system is Filippov and the origin is reached in finite time.
	\end{itemize}
	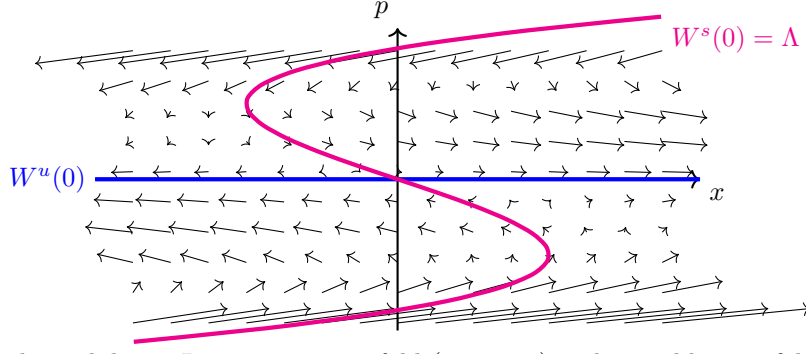
\begin{figure}
		\centering
		\begin{tikzpicture}
			\draw[->, thick] (-4, 0) -- (4, 0);
			\draw[->, thick] (0, -2) -- (0, 2);
			\node[below right] at (4, 0) {$x$};
			\node[above left] at (0, 2) {$p$};
			\foreach \x in {-3.5,-3,...,3.5}{
				\foreach \y in {-1.9,-1.5,...,1.9}{
					\pgfmathsetmacro{\vx}{\x-4*\y*\y*\y+6*\y}
					\pgfmathsetmacro{\vy}{-\y}
					\draw[->] (\x,\y) -- (\x+\vx/10,\y+\vy/10);
				}
			}
			\draw[ultra thick, blue] (-4,0) -- (4,0);
			\draw[ultra thick, magenta, domain=-2.15:2.15, smooth] plot (\x*\x*\x-3*\x, \x);
			\node[below right, magenta] at (3.5,2.15) {$W^s(0) = \Lambda$};
			\node[left, blue] at (-4,0) {$W^u(0)$};
		\end{tikzpicture}
		\caption{The stabilizing Lagrangian manifold (magenta) and unstable manifold (blue).}
		\label{fig:cubic_hysteresis}
	\end{figure}
	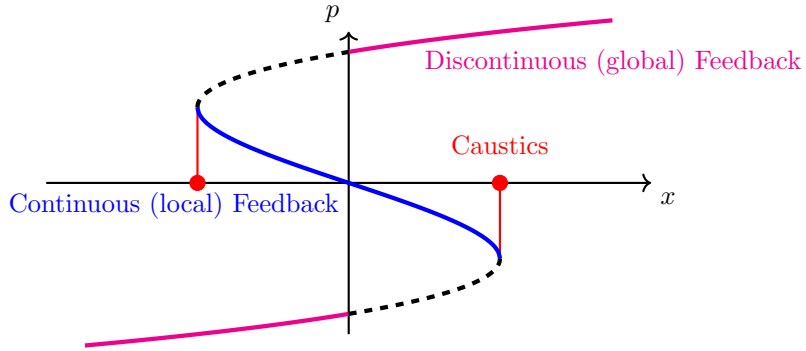
\begin{figure}
		\centering
		\begin{tikzpicture}
			\draw[->, thick] (-4, 0) -- (4, 0);
			\draw[->, thick] (0, -2) -- (0, 2);
			\node[below right] at (4, 0) {$x$};
			\node[above left] at (0, 2) {$p$};
			\draw[ultra thick, magenta, domain=1.7321:2.15, smooth] plot (\x*\x*\x-3*\x, \x);
			\draw[ultra thick, magenta, domain=-2.15:-1.7321, smooth] plot (\x*\x*\x-3*\x, \x);
			\draw[red, thick] (-2,1) -- (-2,0);
			\draw[red, thick] (2,-1) -- (2,0);
			\draw[red, fill] (-2,0) circle [radius=0.1];
			\draw[red, fill] (2,0) circle [radius=0.1];
			\draw[ultra thick, blue, domain=-1:1, smooth] plot (\x*\x*\x-3*\x, \x);
			\draw[ultra thick, dashed, domain=1:1.7321, smooth] plot (\x*\x*\x-3*\x, \x);
			\draw[ultra thick, dashed, domain=-1.7321:-1, smooth] plot (\x*\x*\x-3*\x, \x);
			\node[magenta, below] at (3.5,1.9) {Discontinuous (global) Feedback};
			\node[blue, below left] at (0,0) {Continuous (local) Feedback};
			\node[red] at (2, 0.5) {Caustics};
		\end{tikzpicture}
		\caption{Using the cost to synthesize a stabilizing feedback law for the dynamics \eqref{eq:cubic_hysteresis_dynamics} yields the law given above in magenta. The feedback in blue is continuous, but cannot be extended to a global feedback law.}
		\label{fig:cubic_hysteresis_decomposition}
	\end{figure}
	\subsection{The nonholonomic integrator}\label{sec:nh_integrator}
	Consider the nonholonomic integrator from Example \ref{ex:nh_integrator},
	\begin{equation*}
		\dot{x} = u, \quad \dot{y} = v, \quad \dot{z} = uy - vx,
	\end{equation*}
	along with the running cost
	\begin{equation*}
		\ell =  \frac{1}{2}(x^2+y^2+z^2+u^2+v^2).
	\end{equation*}
	The induced Hamiltonian is
	\begin{equation*}
		H = \frac{1}{2}p_x^2 + yp_xp_y + \frac{1}{2}p_y^2 - xp_yp_z + \frac{1}{2}(x^2+y^2)p_z^2 - \frac{1}{2}\left(x^2+y^2+z^2\right).
	\end{equation*}
	As no continuous feedback law exists stabilizing the origin, the stable manifold cannot project homeomorphically around the origin. This problem can be simplified as $H$ is invariant under rotations of the $xy$-plane resulting in the conserved quantity $P=yp_x-xp_y$. Applying the coordinate transform 
	\begin{equation*}
		x = r\cos\theta, \quad y = r\sin\theta,
	\end{equation*}
	and applying the constraint $P=0$, we obtain the reduced Hamiltonian
	\begin{equation}\label{eq:reduced_nh_hamiltonian}
		H_{red} = \frac{1}{2}r^2 p_z^2 + \frac{1}{2}p_r^2 - \frac{1}{2}\left(z^2 + r^2\right).
	\end{equation}
	\begin{remark}
		The reduced Hamiltonian \eqref{eq:reduced_nh_hamiltonian} is equivalent to the Hamiltonian induced by the following control problem:
		\begin{equation}\label{eq:reduced_nh_integrator}
			\dot{r} = u, \quad \dot{z} = rv,
		\end{equation}
		with cost function
		\begin{equation*}
			J = \int_0^\infty \, \frac{1}{2}(r^2 + z^2 + u^2 + v^2) \, dt.
		\end{equation*}
		It is interesting to note that \eqref{eq:reduced_nh_integrator} does satisfy Brockett's condition even though the stable manifold is singular at the origin, see Fig. \ref{fig:planar_caustics} below.
	\end{remark}
	
	The resulting Hamiltonian dynamics are
	\begin{equation}\label{eq:deg_hamiltonian}
		\begin{array}{lcl}
			\dot{r} = p_r, && \dot{p}_r = r(1-p_z^2), \\
			\dot{z} = r^2p_z, && \dot{p}_z = z.
		\end{array}
	\end{equation}
	As the origin is to be stabilized, the Lagrangian manifold of interest is actually
	\begin{equation*}
		\Lambda = W^s(N), \quad N = \left\{ (0,0,0,p_z) : |p_z|<1 \right\},
	\end{equation*}
	where $N$ is an isotropic manifold, recall Proposition \ref{prop:closed_Lagrangian}.
	For a single point on the $p_z$-axis, the stable eigenspace is
	\begin{equation*}
		E^s([0, 0, 0, p_z]) = \left[ -\sqrt{1-p_z^2}, 0, 1-p_z^2, 0 \right] \cdot \mathbb{R}, \quad |p_z|<1.
	\end{equation*}
	As the manifold $N\subset T^*_{(0,0)}\mathbb{R}^2 \subset T^*\mathbb{R}^2$ is isotropic, by Theorem \ref{thm:att_man_iso} its stable manifold is also isotropic and hence Lagrangian as it is two-dimensional.
	
	A point on a characteristic lies on the caustic if
	\begin{equation*}
		\det J = 0, \quad J = \begin{bmatrix}
			\dot{x} & \Phi_{14} \\
			\dot{y} & \Phi_{24}
		\end{bmatrix},
	\end{equation*}
	where $\Phi$ is the state-transition matrix arising from \eqref{eq:deg_hamiltonian}, i.e.
	\begin{equation*}
		\dot{\Phi} = A(t)\Phi, \quad A = \begin{bmatrix}
			0 & \hspace{0.25cm} 0 & \hspace{0.3cm} 1 & \hspace{0.1cm}0 \\
			2rp_z & \hspace{0.25cm} 0 & \hspace{0.3cm} 0 & \hspace{0.1cm}r^2 \\
			1-p_z^2 & \hspace{0.25cm} 0 & \hspace{0.3cm} 0 &\hspace{0.1cm}-2rp_z \\
			0 & \hspace{0.25cm} 1 & \hspace{0.3cm} 0 &\hspace{0.1cm} 0
		\end{bmatrix}.
	\end{equation*}
	It turns out that there exists an infinite family of cusps emanating from the origin and coalescing to the $z$-axis, see Fig. \ref{fig:planar_caustics}. The presence of these caustics inhibits the construction of a continuous feedback law for the reduced control system \eqref{eq:reduced_nh_integrator}. A discontinuous feedback law can be constructed by choosing which law is implemented within the first caustics, see Fig. \ref{fig:planar_controls}.
	
	\begin{figure}
		\centering
		\begin{subfigure}{0.475\textwidth}
			\centering
			\includegraphics[width=\linewidth]{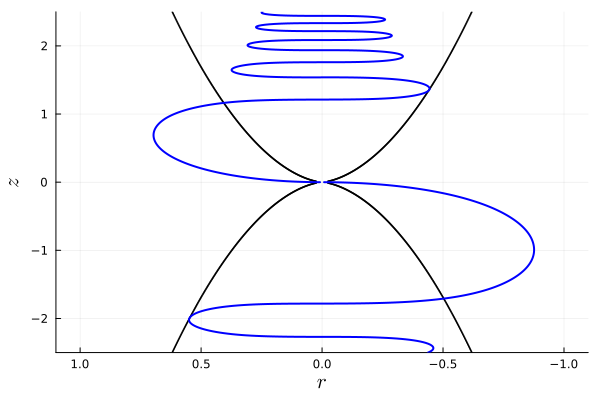}
			\caption{The initial caustic.}
		\end{subfigure}
		\begin{subfigure}{0.475\textwidth}
			\centering
			\includegraphics[width=\linewidth]{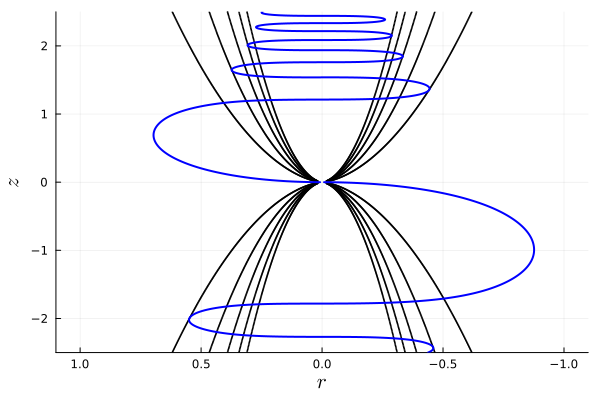}
			\caption{The first five caustics.}
		\end{subfigure}
		\caption{Caustics in the stable manifold for \eqref{eq:deg_hamiltonian}. There is an infinite family of caustics/cusps that accumulate along the $z$-axis that are generated by the ``wiggles'' in the trajectories.}
		\label{fig:planar_caustics}
	\end{figure}
	\begin{figure}
		\centering
		\includegraphics[width=0.8\textwidth]{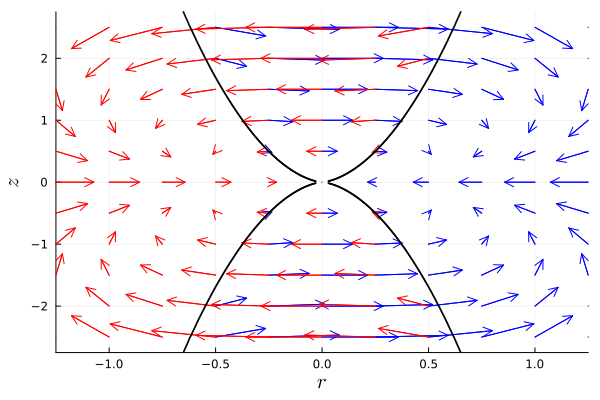}
		\caption{The stabilizing dynamics for \eqref{eq:deg_hamiltonian}. There are two different modes corresponding to the projection of the stable manifold being multi-valued. The boundary of these regions are the caustics.}
		\label{fig:planar_controls}
	\end{figure}
	
	The presence of these caustics, however, does not directly obstruct the existence of a feedback law to the original system as a continuous law can be found for the reduced system under the restriction that $r>0$. To apply this law to the original system, the reconstruction for the controls is
	\begin{equation*}
		\begin{split}
			u &= p_r\cos\theta + yp_z, \\
			v &= p_r\sin\theta - xp_z.
		\end{split}
	\end{equation*}
	As $p_r\ne 0$ on the $z$-axis, this feedback is not continuous. Therefore, as expected, this procedure fails to produce a continuous feedback for the nonholonomic integrator.
	\subsection{Stabilizing a periodic orbit}\label{sec:periodic_orbit}
	For the final example in this work, we will stabilize a periodic orbit as an application of Theorem \ref{thm:main_result} that is not covered by Theorem \ref{thm:generating_fixed}.
	Consider the following control system:
	\begin{equation}\label{eq:unstable_po}
		\begin{split}
			\dot{x} &= -2\pi y + x(x^2+y^2-1) + u, \\
			\dot{y} &= 2\pi x + y(x^2+y^2-1).
		\end{split}
	\end{equation}
	When $u=0$, the unit circle is an unstable limit cycle given by
	\begin{equation*}
		\Gamma : \gamma(t) = \left( \cos 2\pi t, \sin 2\pi t\right).
	\end{equation*}
	
	To stabilize this orbit, we choose the cost
	\begin{equation*}
		\ell = \frac{1}{2}u^2 + \frac{1}{2}(x^2+y^2-1)^2,
	\end{equation*}
	which generates the Hamiltonian
	\begin{equation*}
		\begin{split}
			H &= p_x\dot{x} + p_y\dot{y} - \frac{1}{2}u^2 - \frac{1}{2}(x^2+y^2-1)^2,
			\quad u = p_x.
		\end{split}
	\end{equation*}
	The resulting equations of motion are
	\begin{equation}\label{eq:hamilton_po}
		\begin{split}
			\dot{x} &= -2\pi y + x(x^2+y^2-1) + p_x, \\
			\dot{y} &= 2\pi x + y(x^2+y^2-1), \\
			\dot{p}_x &= 2x(x^2+y^2-1) - p_x(3x^2+y^2-1) - p_y(2\pi+2xy), \\
			\dot{p}_y &= 2y(x^2+y^2-1) - p_x(2xy-2\pi) - p_y(x^2+3y^2-1).
		\end{split}
	\end{equation}
	The isotropic manifold is
	\begin{equation}\label{eq:isotropic_po}
		Z_\Gamma = \left\{ \left( \cos 2\pi t, \sin 2\pi t, 0, 0\right) : 0\leq t < 1 \right\}.
	\end{equation}
	To determine the stable manifold, $W^s(Z_\Gamma)$, and to apply Theorem \ref{thm:main_result}, we linearize about the orbit. The Jacobian of \eqref{eq:hamilton_po} is
	\begin{equation*}
		A(t) = \begin{bmatrix}
			2\cos^2 2\pi t & \sin(4\pi t)-2\pi & 1 & 0 \\
			\sin(4\pi t)+2\pi & 2\sin^2 2\pi t & 0 & 0 \\
			4\cos^2 2\pi t & 2\sin(4\pi t) & -2\cos^2 2\pi t & -\sin(4\pi t)-2\pi\\
			2\sin(4\pi t) & 4\sin^2 2\pi t & 2\pi-\sin(4\pi t) & -2\sin^2 2\pi t
		\end{bmatrix}.
	\end{equation*}
	The variational equation $\dot{\Phi}=A(t)\Phi$ with $\Phi(0)=\mathrm{Id}$ results in the state-transition matrix:
	\begin{equation*}
		\Phi(1) \approx \begin{bmatrix}
			10.406 & 0.0 & 1.2936 & -0.3594\\
			0.3504 & 1.0 & 0.00902 & 0.4835 \\
			9.1707 & 0.0 & 1.2361 & -0.3504 \\
			0.0 & 0.0 & 0.0 & 1.0
		\end{bmatrix}.
	\end{equation*}
	The four eigenvalues of this matrix are
	\begin{equation*}
		\sigma\left( \Phi(1) \right) \approx \left\{
		0.0865, 1.0, 1.0, 11.556 \right\}.
	\end{equation*}
	Therefore, $\Phi(1)$ is a hyperbolic matrix and $Z_\Gamma$ is a symplectic saddle. The stable eigenvector is
	\begin{equation*}
		v^s \approx \begin{bmatrix}
			-0.1243 \\
			0.0379 \\
			0.9915 \\
			0.0
		\end{bmatrix},
	\end{equation*}
	which is not vertical which implies the existence of a feedback law. To visualize the stable manifold, we restrict to the Poincar\'{e} section $\{y=0\}$ and apply the energy constraint $H=0$. The stable manifold in the $(x,p_x)$-coordinates is shown in Fig. \ref{fig:orbit_stable_manifold}. Interestingly, the periodic orbit is homoclinic and it appears that the homoclinic loop encloses a family of quasi-periodic orbits, rendering the system integrable. The synthesized feedback law is shown in Fig. \ref{fig:stabilizing_po}.
	\begin{figure}
		\centering
		\begin{subfigure}{0.475\textwidth}
			\centering
			\includegraphics[width=\linewidth]{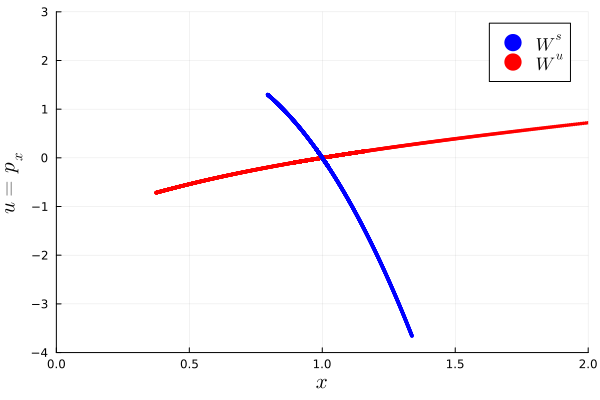}
			\caption{The local stable and unstable manifolds.}
		\end{subfigure}
		\begin{subfigure}{0.475\textwidth}
			\centering
			\includegraphics[width=\linewidth]{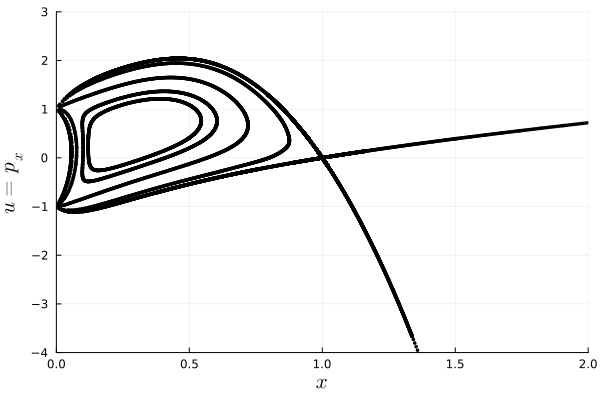}
			\caption{The homoclinic loop.}
		\end{subfigure}
		\caption{The stable and unstable manifolds of the isotropic manifold \eqref{eq:isotropic_po} under the dynamics \eqref{eq:hamilton_po} restricted to Poincar\'{e} section and an energy surface. As the periodic orbit at $x=1$ is not in the caustic for $W^s(Z_\Gamma)$, a feedback law can be synthesized.}
		\label{fig:orbit_stable_manifold}
	\end{figure}
	\begin{figure}
		\centering
		\includegraphics[width=0.8\textwidth]{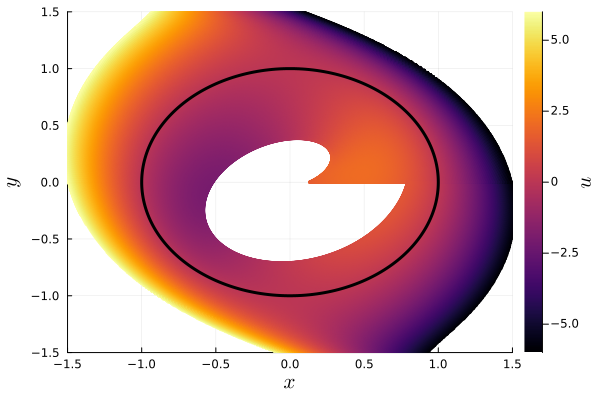}
		\caption{The stabilizing feedback law computed from $W^s(Z_\Gamma)$ for \eqref{eq:unstable_po}.}
		\label{fig:stabilizing_po}
	\end{figure}
	
	\section{Conclusion}\label{sec:conclusion}
	Feedback stabilization can be viewed through the context of Hamiltonian systems and dynamical systems theory. By using Pontryagin's maximum principle, an infinite-horizon control problem is transformed into finding stable (Lagrangian) submanifolds. Obstructions to stabilization translate to geometric features of these submanifolds, e.g., failing to project diffeomorphically to the base and sections being multi-valued. 
	
	We point out two avenues for further study:
	\begin{enumerate}
		\item Target sets in control theory correspond to isotropic submanifolds of a Hamiltonian system. This creates two purely dynamical questions (recall Problem \ref{quest:LAP}):
		\begin{center}
			Let $(M,\omega, H)$ be a Hamiltonian system and $N\subset M$ an isotropic submanifold. Is $(N,X_H)$ a symplectic saddle? 
		\end{center}
		The answer to this question requires more delicate analysis than the standard stable manifold theorem as $N\subset M$ cannot be hyperbolic, Theorem \ref{thm:isotropic_nhim}. A follow up question is the following:	
		\begin{center}
			Suppose that $(N,X_H)$ is a symplectic saddle and $M=T^*Q$ is a cotangent bundle. Does $N$ belong to the stable caustic set? i.e. are the following sets disjoint
			\begin{equation*}
				\pi_Q(N) \cap \Sigma \left(W^s\left( N \right)\right) = \emptyset.
			\end{equation*}
		\end{center}
		Resolving this second question provides sufficient conditions for the existence of feedback laws.
		\item This work exclusively studied time-invariant feedback laws through infinite-horizon optimal control. A natural extension is for time-varying feedback. This results in solving the boundary problem \eqref{eq:mom_boundary} from Theorem \ref{thm:pmp}. Three notable final conditions are
		\begin{description}
			\item[Fixed point:] If it is required that $x(t_f) = x_f$, the final momentum is
			\begin{equation*}
				p(t_f) \in T_{x_f}^*Q \subset T^*Q.
			\end{equation*}
			\item[Constraint set:] If it is required that $x(t_f)\in N$ where $N\subset Q$ is a submanifold, then
			\begin{equation*}
				p(t_f) \in \Lambda_N := \mathrm{Ann}(TN)\subset T^*Q.
			\end{equation*}
			\item[Terminal cost:] If there is a cost penalty of $g(x(t_f))$, then
			\begin{equation*}
				p(t_f) \in \Gamma_{dg} = \left\{ dg_x : x\in Q \right\} \subset T^*Q.
			\end{equation*}
		\end{description}
		All three of these sets, $T_{x_f}^*Q$, $\Lambda_N$, and $\Gamma_{dg}$, are Lagrangian submanifolds of $T^*Q$ - but are no longer invariant under the Hamiltonian flow. Additionally, for the initial conditions, only $x(t_0)$ is specified while $p(t_0)$ is free. This means that the initial conditions lie within the Lagrangian submanifold $T_{x_0}^*Q$. If $\varphi_t$ is the Hamiltonian flow, then solutions to the boundary value problem are characterized by the intersection
		\begin{equation*}
			T_{x_0}^*Q \cap \varphi_{-t}\left( \Lambda \right),
		\end{equation*}
		where $\Lambda$ is the terminal Lagrangian submanifold. Therefore, the construction of time-varying feedback laws is related to the intersection of Lagrangian manifolds under Hamiltonian isotopies. Estimates on Lagrangian intersections is a central aspect of symplectic topology, e.g., \cite[Chapter~11]{mcduff} for a pleasant overview. Unfortunately, much of the theory developed requires that the symplectic manifold be compact while cotangent bundles are never compact (cf. the Arnold-Givental conjecture).
	\end{enumerate}
	
	\backmatter 
	\bmhead{Acknowledgements}
	This work was funded by AFOSR grant FA9550-23-1-0400.

	\bibliography{sn-bibliography}%
	
\end{document}